\documentclass[12pt]{article}
\usepackage{no-ejc}
\usepackage{mathtools}
\usepackage{algorithm2e}

\usepackage{amsmath,amssymb}
\usepackage{enumerate}
\usepackage{graphicx}
\usepackage{tikz}
\usepackage{comment}

\usepackage[colorlinks=true,citecolor=black,linkcolor=black,urlcolor=blue]{hyperref}

\title{\bf An Eventown Result for Permutations}

\author{
Nathan Lindzey\\
\small Department of Mathematical Sciences\\[-0.8ex]
\small University of Memphis\\[-0.8ex] 
\small\tt nathan.lindzey@memphis.edu
}

\dateline{}{}{}


\begin{document}

\maketitle
\frenchspacing

\begin{abstract}
A family of permutations $\mathcal{F} \subseteq S_n$ is \emph{even-cycle-intersecting} if $\sigma \pi^{-1}$ has an even cycle for all $\sigma,\pi \in \mathcal{F}$. 
We show that if $\mathcal{F} \subseteq S_n$ is an even-cycle-intersecting family of permutations, then $|\mathcal{F}| \leq 2^{n-1}$, and that equality holds when $n$ is a power of 2 and $\mathcal{F}$ is a double-translate of a Sylow 2-subgroup of $S_n$. This result can be seen as an analogue of the classical eventown problem for subsets and it  confirms a conjecture of J\'anos K\"orner on maximum reversing families of the symmetric group. Along the way, we show that the canonically intersecting families of $S_n$ are also the extremal  \emph{odd-cycle-intersecting} families of $S_n$ for all even $n$. While the latter result has less combinatorial significance, its proof uses an interesting new character-theoretic identity that might be of independent interest in algebraic combinatorics. 
\end{abstract}

\section{Introduction}

An \emph{intersecting family} is a subset $\mathcal{F} \subseteq \mathcal{X}$ of some domain $\mathcal{X}$ such that any two elements of $\mathcal{F}$ intersect non-trivially. The Erd\H{o}s--Ko--Rado theorem states that if $\mathcal{X}$ is the collection of all $k$-sets of an $n$-element set, then for $k < n/2$, an intersecting family $\mathcal{F} \subseteq \mathcal{X}$ has size no greater than $\binom{n-1}{k-1}$, and that equality holds if and only if $\mathcal{F}$ is a family obtained taking all $k$-sets that contain some fixed element $i \in \{1,2,\ldots,n\} =: [n]$. Many other seminal results in extremal combinatorics have since appeared under the umbrella of so-called \emph{Erd\H{o}s--Ko--Rado combinatorics} (see~\cite{GodsilMeagher} for a comprehensive account).

Recently, the study of intersecting families in finite groups has led to several interesting developments in extremal combinatorics~\cite{KupavskiiZ24,KellerLMS24,ErnstS23}. Many types of intersection have been studied (e.g., $t$-intersecting~\cite{EllisFP11}, setwise-intersection~\cite{Ellis11}, $\lambda$-intersection~\cite{DafniFLLV26}). In each of these cases, the notion of intersection is defined with respect to combinatorial properties of the \emph{difference} $gh^{-1}$ between two group elements $g,h \in G$. Groups of course possess more than just combinatorial structure, so one may try to further understand group structure by studying subsets that are extremal with respect to a pairwise algebraic condition. In this vein, we propose an algebraic notion of intersection that depends on the \emph{order} of $gh^{-1}$, which we model as non-adjacency in the following natural graph.

\begin{definition}[The $p$-regular graph of $G$]
 Let $G$ be a finite group and $p$ be a prime dividing the order of $G$. An element $g \in G$ is \emph{$p$-regular} if $|g| \not \equiv 0 \mod p$. Let $G_{p'} \subseteq G$ be its set of $p$-regular elements. We define the \emph{$p$-regular graph} $\Gamma_{p'}(G) := \mathrm{Cay}(G,G_{p'} \! \setminus \! \{\mathrm{id}\})$ to be the Cayley graph of $G$ generated by its non-identity $p$-regular elements $G_{p'} \! \setminus \! \{\mathrm{id}\}$.
\end{definition}

Two elements $g,h \in G$ are non-adjacent in $\Gamma_{p'}(G)$ if $gh^{-1}$ is \emph{$p$-singular}, i.e., $|gh^{-1}| \equiv 0 \mod p$. The independent sets of $\Gamma_{p'}(G)$ are families of pairwise $p$-singular elements of $G$, which we call \emph{(pairwise) $p$-singular families} of $G$. A \emph{maximum $p$-singular family} of $G$ is a $p$-singular family of $G$ of largest size. By design, the \emph{independence number} of $\Gamma_{p'}(G)$, denoted as $\alpha(\Gamma_{p'}(G))$, is the size of a maximum $p$-singular family of $G$.

Since $G_{p'}$ is closed under conjugation by $G$, the $p$-regular graph is a normal Cayley graph; therefore, its eigenvalues are determined by the character theory of $G$ (see~\cite{Diaconis88}, for example). 
Erd\H{o}s and Tur\'an~\cite{ErdosT67b} obtained explicit formulas for the largest eigenvalue (i.e., degree) of $\Gamma_{p'}(G)$ when $G = S_n$ is the \emph{symmetric group} on $n$ symbols. 
Siemons and Zalesski~\cite{SiemonsZ19} consider the complement of $\Gamma_{p'}(G)$, i.e., the Cayley graph generated by its $p$-singular elements $G \setminus G_{p'}$. They were interested in the spectra of such graphs in relation to the problem of constructing Cayley graphs with a singular adjacency matrix. Ebrahimi~\cite{Ebrahimi26} also investigated this graph for $p$-solvable groups in a similar vein. 

Besides being a natural variation of the Erd\H{o}s--Ko--Rado problem for groups, these questions seem to have strong connections to Sylow theory, as one might expect. For example, consider the group $S_n$ and let $p=2$. We say two permutations $\sigma,\pi$ are \emph{even-cycle-intersecting} if $\sigma \pi^{-1}$ has an even cycle, equivalently, $|\sigma \pi^{-1}| \equiv 0 \mod 2$.  Our main result is a tight bound on the size of a maximum $2$-singular family of $S_n$ when $n = 2^\ell$.

\begin{theorem}[Main Result]\label{thm:main}
If $\mathcal{F} \subseteq S_n$ is even-cycle-intersecting, then $|\mathcal{F}| \leq 2^{n-1}$ for all $n \geq 2$.
The bound is sharp when $\mathcal{F}$ is a Sylow 2-subgroup and $n = 2^\ell$ for some $\ell \in \mathbb{N}$.
\end{theorem} 
\noindent The main result implies a symmetric group analogue of the \emph{eventown problem} for subsets. A family $\mathcal{F} \subseteq 2^{[n]}$ is said to be an \emph{eventown} if $| \pi | \equiv 0 \mod 2$ for each $\pi \in \mathcal{F}$ and $| \sigma \cap \pi | \equiv 0 \mod 2$ for all $\sigma, \pi \in \mathcal{F}$ (see~\cite[Ch.~1]{BF92}, for example). 
Answering a question of Erd\H{o}s, Berlekamp~\cite{Berlekamp69} showed that the maximum size of an eventown is $2^{\lfloor n/2 \rfloor}$, and that equality is attained by taking all subsets $S \in 2^{[n]}$ such that $2i \in S \Leftrightarrow 2i-1 \in S$ for all $1 \leq i \leq n/2$. One can define a family $\mathcal{F} \subseteq S_n$ to be an \emph{eventown} if $| \pi | \equiv 0 \mod 2$ for non-trivial $\pi \in \mathcal{F}$ and $| \sigma \pi^{-1} | \equiv 0 \mod 2$ for all $\sigma, \pi \in \mathcal{F}$. Here, size and intersection are naturally replaced by order and difference, and we see that a Sylow 2-subgroup is an eventown of maximum size provided that $n$ is a power of 2. 

A family $\mathcal{F} \subseteq S_n$ is \emph{2-set-intersecting} if and only if for any two
permutations $\sigma, \pi \in \mathcal{F}$, there exists $\{i, j\} \subseteq [n] := \{1,2,\ldots,n\}$ such that one of the following conditions holds: 
\begin{align}
\sigma(i) = \pi(i) \quad &\text{   and   } \quad  \sigma(j) = \pi(j) \\
 \quad \sigma(i) = \pi(j) \quad &\text{  and   } \quad \sigma(j) = \pi(i)~.
\end{align}
Answering a question of K\"orner, Ellis showed for $n$ sufficiently large, that the largest 2-set-intersecting families in $S_n$ are the double-translates of the Young subgroup $S_{n-2} \times S_2$. This was later shown to hold for all $n \geq 6$ by Meagher and Razafimahatratra, and Chase, Dafni, Filmus, and Lindzey~\cite{MeagherR21,2perm}.

If we stipulate that only condition~(1) above occurs, then
we are insisting that the family is \emph{2-intersecting}. Ellis, Friedgut, and Pilpel showed for 
sufficiently large $n$ that the largest 2-intersecting families in $S_n$ are the double-translates of the Young subgroup $S_{n-2} \times S_1 \times S_1$. This too was shown to hold for all $n \geq 5$~\cite{MeagherR21,2perm}.

K\"orner asks what happens if we stipulate that only condition (2) occurs. Here, we say a family $\mathcal{F} \subseteq S_n$ is \emph{reversing}  if condition (2) holds for any $\sigma, \pi \in \mathcal{F}$. K\"orner conjectured that such a family cannot have size greater than $C^n$ for some absolute constant $C$, and gave an eventown-type construction for a reversing family of size $2^{\lfloor n/2 \rfloor}$ by taking a largest elementary Abelian 2-subgroup of $S_n$. F\"{u}redi, Kantor, Monti, and Sinaimeri~\cite{FurediKMS10,FurediKMS11} gave a lower bound of $8^{\lfloor n/5\rfloor}$which Harcos and Solt\'esz~\cite{HarcosS20} later improved to $4^{\lfloor n/3 \rfloor}$. F\"{u}redi et al.~\cite{FurediKMS10,FurediKMS11} also considered \emph{reverse-free} families of permutations, i.e., families of permutations such that condition (2) does not hold for any pair of members, wherein they also pose the question of determining the size of a maximum reversing family. Leveraging the fact that constructions of reverse-free families give rise to upper bounds on reversing families, Cibulka~\cite{Cibulka13} improved the results of F\"{u}redi et al., obtaining an upper bound of $n^{n/2 + O(n/\log_2 n)}$ on the size of a maximum reversing family. 
We refer the reader to~\cite{HarcosS20,CohenFK19,ByrneT24} for other lines of combinatorial inquiry surrounding K\"orner's question.

Since a reversing family is even-cycle-intersecting, Theorem~\ref{thm:main} proves the following conjecture of K\"orner communicated by Ellis~\cite{ellis2019,Ellis11} with $C = 2$.
\begin{theorem}\label{con:korner}\emph{\cite[Conjecture~4]{Ellis11}}
There exists an absolute constant $C$ such that for any $n \in \mathbb{N}$, a reversing family in $S_n$ has size at most $C^n$.
\end{theorem}
\noindent It is worth noting that the resolution of K\"orner's conjecture has a few other combinatorial implications, see~\cite[\S 5]{HarcosS20} for these details. We find Theorem~\ref{thm:main} and its proof to be somewhat surprising for the following reasons. 

Let $S$ be the union of all Sylow 2-subgroups of $S_n$. The order of any element of $S$ is a power of 2, a much more restrictive condition than just being 2-singular. Erd\H{o}s and Tur\'an~\cite{ErdosT67b} showed the probability drawing a $2$-singular element of $S_n$ uniformly at random is at least $(1 - 3/\sqrt{n})$.
The foregoing shows that despite the preponderance of 2-singular elements, we have
$\alpha(\Gamma_{2'}(S_{2^\ell})) = \alpha(\mathrm{Cay}(S_{2^\ell}, \bar{S}))$, even though $\Gamma_{2'}(S_{2^\ell}) \subset \mathrm{Cay}(S_{2^\ell}, \bar{S})$ is a significantly sparser graph. One can compare this situation to Mantel's theorem, in that the seemingly more restrictive condition of forbidding odd cycles turns out to be just as restrictive as forbidding triangles. 

In extremal combinatorics, when the extremal families consist of all objects containing a fixed copy of some substructure, then we say such families are of \emph{kernel-type}, also known as \emph{dictators}, \emph{juntas}, and \emph{trivially (canonically) intersecting families} in Erd\H{o}s--Ko--Rado combinatorics. A common theme in Erd\H{o}s--Ko--Rado combinatorics is that for sufficiently large $n$, the largest $t$-intersecting families are of kernel-type. Eigenvalue techniques have played a particularly crucial role in the development of this area~\cite{GodsilMeagher}. Theorem~\ref{thm:main} is a rare example where eigenvalue techniques produce exact results for an intersection problem where the extremal families are not of kernel-type.

For primes $p \neq 2$, one can verify for small $n$ that the Sylow $p$-subgroups of $S_n$ are not always maximum $p$-singular families, which might seem unexpected if it weren't for the fact that the Sylow $2$-subgroups of symmetric groups often possess exceptional properties. Indeed, there has been a recent flurry of activity concerning their properties. For example, Diaconis, Giannelli, Guralnick, Law, Navarro, Sambale, and Spink~\cite{DiaconisGGLNSS25} show that a pair of random Sylow $p$-subgroups of $S_n$ almost always intersect trivially for $p \neq 2$ and $n \rightarrow \infty$, whereas if $p = 2$, then such pairs intersect non-trivially with probability at least $(1-\sqrt{e}+o(1))$. Eberhard~\cite{Eberhard25} later gave a matching upper bound and Renteln~\cite{Renteln25} found a good algorithm for counting Sylow $p$-subgroup double cosets of $S_n$. 

As we prove our main result, we point out along the way some perhaps overlooked connections between bijective combinatorics and character theory. For example, we prove a new character-theoretic identity (Theorem~\ref{thm:2row}) that is complementary to a result of Regev~\cite[Proposition 1.1]{Regev13}, leading to a characterization of the extremal \emph{odd-cycle-intersecting} families of $S_n$. Our results also have some connection to \emph{Steinberg-like characters} recently introduced by Malle and Zalesski~\cite{MalleZ20}. We conclude with a few conjectures, some open problems, and directions for future work.

\section{Preliminaries}
We give a brief overview of some standard bounds on the independence numbers of graphs. For any graph $X = (V,E)$, it is well known that the \emph{Lov\'asz $\vartheta$-function} $\vartheta(X)$ gives an SDP relaxation of the independence number of $X$:
\begin{align}\label{eq:sdp}
	\alpha(X) \leq \vartheta(X) = 
	\begin{cases}
	\min \theta \\
	\text{$A \succeq 0$}, \quad 	A \in \mathrm{Mat}^{n \times n}(\mathbb{R}) \\
	A_{i,i} = \theta - 1 \text{ for every } i \in V\\
	A_{i,j} = -1 \text{ if $i \not \sim j$}~.
	\end{cases}
\end{align}
To prove Theorem~\ref{thm:main}, we solve Program (\ref{eq:sdp}) for the 2-regular graph of $S_n$ for all $n$ that are powers of 2 (in fact, we give an integral extreme point of the associated spectahedron). 

If $X$ is a regular graph, then we can obtain a spectral upper bound on $\vartheta(X)$ via the following well-known result of Delsarte and Hoffman. 
\begin{theorem}[Delsarte--Hoffman]\label{thm:ratioInd}
Let $A$ be a weighted adjacency matrix of a regular graph $X = (V,E)$ with constant row sum. Let $\eta_{\max}$ and $\eta_{\min}$ be the greatest and least eigenvalue of $A$ respectively. Then 
$$\alpha(X) \leq  \vartheta(X) \leq |V|  \left( \frac{-\eta_{\min} }{\eta_{\max}-\eta_{\min}} \right).$$
Moreover, if equality holds, then the characteristic vector $1_S \in \mathbb{R}^{V}$ of $S$ lives in the direct sum of the eigenspaces of $A$ corresponding to $\eta_{\max}$ and $\eta_{\min}$.
\end{theorem}
If $X = \mathrm{Cay}(G,S)$ is a normal Cayley graph, i.e., its generating set $S \subseteq G$ is closed under conjugation, then by Schur's lemma, we can determine its eigenvalues by evaluating irreducible characters (see~\cite{Diaconis88}, for example).
\begin{theorem}
Let $G$ be a finite group and $\{\chi^{(i)}\}_{i=1}^m$ be the set of irreducible characters of $G$. Then the eigenvalues of the adjacency matrix of a normal Cayley graph $X = \mathrm{Cay}(G, S)$ are given by
$$\eta^{(i)} = \frac{1}{\chi^{(i)}(1)} \sum_{s \in S} \chi^{(i)}(s) \quad \text{ for all } i = 1,\ldots,m$$ 
where the multiplicity of $\eta^{(i)}$ is $\chi^{(i)}(1)^2$.
\end{theorem}
\noindent When $G$ is a symmetric group, its characters can be computed combinatorially as follows. Let $\chi^\lambda$ denote the character of $S_n$ corresponding to the shape $\lambda \vdash n$. A skew shape $\lambda / \nu$ is a \emph{border strip} if it is connected and does not contain the shape $\boxplus$. The \emph{leg-length} $ll(\lambda / \nu)$ of a border strip $\lambda / \nu$ is the number of rows minus 1.
\begin{theorem}[Murnaghan--Nakayama]\label{thm:mn} Let $\sigma \in S_n$ be a permutation with cycle type $\mu \vdash n$. Let $\alpha = \alpha_1,\alpha_2, \ldots, \alpha_\ell$ be an integer composition of $n$ that is any reordering of $\mu$. Let $\alpha' = \alpha_2, \ldots, \alpha_\ell$ and $\lambda \vdash n$. Then 
$$
	\chi^{\lambda}(\sigma) = \chi^{\lambda}_\mu = \chi^{\lambda}_\alpha = \sum_{\nu} (-1)^{ll(\lambda /\nu)} \chi^{\nu}_{\alpha'}
$$
where the sum ranges over all $\nu$ such that $\lambda / \nu$ is a border strip on $\alpha_1$ cells.
\end{theorem}

Let $G = S_n$, $S = G_{2'} \setminus \{\mathrm{id}\}$, and recall that $\Gamma_{2'}(S_n) = \mathrm{Cay}(G,S) =: X$. The standard approach to prove Theorem~\ref{thm:main} at this juncture would be to use the Murnaghan--Nakayama rule to determine the least eigenvalue of the adjacency matrix of $X$, and observe that Delsarte--Hoffman bound is met with equality. For small $n$, one can verify computationally that this approach does not work, and that Delsarte--Hoffman in fact gives rather poor bounds on the independence number of the 2-regular graph of $S_n$. This leads one to consider the myriad of edge weightings of $X$, and the difficulty lies in finding a weighting such that the eigenvalues of the weighted adjacency matrix meet the Delsarte--Hoffman bound with equality. We find this proverbial needle in a haystack by choosing an edge weighting of $X$ informed by a curious folklore bijection on permutations, that evidently has a strong connection to Sylow 2-subgroups of the symmetric group.

\section{Bijections}\label{sec:bijections}

We begin by recalling a complementary pair of what appear to be folklore identities on permutations, which we record here for posterity (see~\cite{Chen24, HopkinsMO, YuanMO} for more discussion). The first identity concerns elements of the symmetric group $S_n$ that have only odd cycles: 
\begin{align}\label{eq:odd}
	n! = \sum_{\substack{\sigma \in S_n \\ |\sigma_i| \not \equiv 0 \text{ mod } 2~\forall i}} \!\!\!\! 2^{\#\mathrm{cyc}(\sigma)-1}
\end{align}
where $\sigma = \sigma_1\sigma_2\cdots \sigma_m$ are the disjoint cycles that compose $\sigma$ and $\#\text{cyc}(\sigma)=m$. To see this, we canonically order the permutations of $S_n$ as follows. First, cyclically shift each cycle so that the largest element of each cycle appears first. Next, order the disjoint cycles in ascending order according to their first symbol. We say such a permutation is \emph{canonically ordered}. To each canonically ordered $\sigma = \sigma_1\sigma_2\cdots \sigma_m \in S_n$ with only odd cycles, we associate any bitstring $b = b_1b_2\cdots b_{m-1}$ of length $m-1$. We interleave the cycles and bits
\[
\sigma_1 b_1 \sigma_2 b_2 \cdots \sigma_{m-1} b_{m-1} \sigma_m
\]
and process them from left to right according to the following rule: if $b_i = 1$, then remove the last symbol from $\sigma_i$ and place it at the end of $\sigma_{i+1}$; otherwise, we do nothing to $\sigma_i$. The resulting permutation $\pi \in S_n$ is canonically ordered. Moreover, given a canonically ordered permutation $\pi \in S_n$, one can reverse this process to obtain a unique permutation $\sigma$ with $m$ cycles, all of odd length, paired with a unique bitstring $b_1 \cdots b_{m-1}$. This shows the procedure above gives a bijection, and so Equation~(\ref{eq:odd}) follows.

A similar identity holds for elements of $S_n$ consisting of only even cycles, provided that $n$ is even:
\begin{align}\label{eq:even}
	n! =  \sum_{\substack{\sigma \in S_n \\ |\sigma_i| \equiv 0 \text{ mod } 2~\forall i}}  \!\!\!\! 2^{\#\mathrm{cyc}(\sigma)}.
\end{align}
The bijection is the same as before, only now we associate a bit to each cycle.

Equation~(\ref{eq:odd}) can be generalized at the level of characters. Let $h_{k,n} := \chi^{(n-k,1^k)}$ and define $H_n := \sum_{k=0}^{n-1} h_{k,n}$ to be the \emph{sum of the hooks character}. The theorem below follows from the Murnaghan--Nakayama rule; however, it was first shown by Regev~\cite{Regev13} in a more general form using the theory of Lie superalgebras. Taylor~\cite{Taylor17} later gave another proof of Regev's general result using skew characters of the symmetric group, and then Wildon~\cite{WildonBlog} gave a short proof of the theorem below using exterior algebras. We give our elementary proof here for completeness, claiming no originality.
\begin{theorem}\label{thm:hooks} For all $\sigma \in S_n$, we have
	$$
		H_n (\sigma) = 
		\begin{cases}
		2^{\#\mathrm{cyc}(\sigma) - 1} \quad& \text{ if } \sigma \text{ has no even cycle};\\
		0 \quad& \text{ otherwise.}
		\end{cases} 
	$$
\end{theorem}

\begin{proof}[Proof of Theorem~\ref{thm:hooks}]
We proceed by induction on $m := \#\mathrm{cyc}(\sigma)$, the number of cycles.

If $\sigma$ is a single odd cycle, then by Theorem~\ref{thm:mn} we have $H_n(\sigma) = 1$. Similarly, if $\sigma$ is a single even cycle, then $H_n(\sigma) =0$.

\medskip
\noindent Assume $\sigma = \sigma_1\cdots\sigma_m$ has no even cycle, $|\sigma_1| = l$, and $\sigma' := \sigma_2\cdots\sigma_m$. Theorem~\ref{thm:mn} gives
	\[
		H_n(\sigma) = \sum_{k=0}^{n-1} h_{k,n}(\sigma)  = \sum_{k=0}^{n-1} h_{k-l,n-l}(\sigma') + h_{k,n-l}(\sigma')  
	\]
	noting that, by default, if $k < l$, then $h_{k-l,n-l}(\sigma') = 0$, and if $n-k \leq l$, then $h_{k,n-l}(\sigma') = 0$. This gives
	\[
		H_n(\sigma)  = \sum_{k=l}^{n-1} h_{k-l,n-l}(\sigma') + \!\!  \sum_{k=0}^{n-l-1} h_{k,n-l}(\sigma')  
 = 2 \sum_{j=0}^{n-l-1} h_{j,n-l}(\sigma') = 2H_{n-l}(\sigma') = 2^{\#\text{cyc}(\sigma) - 1}
	\]
	where the last equality holds by induction.

\medskip
	
\noindent Assume $\sigma = \sigma_1 \cdots \sigma_m$ has an even cycle $\sigma_1$ without loss of generality. Let $|\sigma_1| = l$ and $\sigma' := \sigma_2,\sigma_3 \cdots, \sigma_m$. By Theorem~\ref{thm:mn} we have
	\[
		H_n(\sigma) = \sum_{k=0}^{n-1}  h_{k,n-l}(\sigma') - h_{k-l,n-l}(\sigma').
	\]
Similarly, this gives $H_n(\sigma) = H_{n-l}(\sigma') - H_{n-l}(\sigma') = 0$, which completes the proof.
\end{proof}
\noindent We are unaware of any previous work that generalizes Equation~(\ref{eq:even}) at the level of characters, so we now prove the result complementary to the theorem above. Here, we let $b_{k,n} := \chi^{(n-k,k)}$ be our shorthand for a two-row character, and let $B_n := \sum_{k=0}^{n/2} (-1)^k b_{k,n}$.
\begin{theorem}\label{thm:2row}  Let $n$ be even. For all $\sigma \in S_n$, we have 
	$$
		B_n (\sigma) = 
		\begin{cases}
		2^{\#\mathrm{cyc}(\sigma)} \quad& \text{ if } \sigma \text{ has no odd cycle};\\
		0 \quad& \text{ otherwise.}
		\end{cases} 
	$$
\end{theorem}
\begin{proof}
We proceed by induction on $m := \#\mathrm{cyc}(\sigma)$, the number of cycles. Let $\sigma$ be a single even cycle. By Theorem~\ref{thm:mn}, we have
$$B_n(\sigma) = b_{0,n}(\sigma) - b_{1,n}(\sigma) + \sum_{k=2}^{n/2 - 2} b_{k,n}(\sigma) = 1 - (-1) + 0 = 2.$$
	
Assume that $\sigma = \sigma_1\sigma_2,\cdots,\sigma_m$ has no odd cycle and $m \geq 2$. Let $|\sigma_1| = l \leq n/2$ without loss of generality, and let $\sigma' := \sigma_2  \cdots, \sigma_m$. 

For each summand $(-1)^kb_{k,n}$ of $B_n$ such that $l \leq n-2k$, remove $l$ cells from the first row to obtain $b_{k,n-l}$. This gives a rim hook of leg length 0, together giving a total contribution of $B_{n-l}(\sigma')$. For each summand $(-1)^kb_{k,n}$ of $B_n$ such that $l \leq k$, remove $l$ cells from the second row to obtain $b_{k-l,n-l}$. This gives a rim hook of leg length 0, which in turn gives a contribution of 
$$\sum_{k=0}^{n/2 - l} (-1)^k b_{k,n-l}(\sigma').$$ 
The only remaining rim hooks to consider are those of leg length 1. The shapes that admit such rim hooks are the $b_{k,n}$'s such that $n-2k+2 \leq l$. Since the leg length is 1, after removing the rim hook, each such summand $(-1)^kb_{k,n}$ of $B_n$ gives $(-1)^{k+1} b_{k',n-l}$ for some $k' \not \equiv k \mod 2$. In particular, we obtain a total contribution of 
$$\sum_{k=n/2 - l + 1}^{(n-l)/2} (-1)^{k} b_{k,n-l}(\sigma').$$ 
Summing up all three cases gives $2B_{n-l}(\sigma') = 2^{m}$ by induction. 

\medskip

Assume $\sigma = \sigma_1\sigma_2,\cdots,\sigma_m$ has an odd cycle $\sigma_1$. Since $n$ is even, we have $l := |\sigma_1| \leq n/2$. Let $\sigma' := \sigma_2  \cdots, \sigma_m$. A similar argument shows that $B_n(\sigma) = 0$. For convenience, define $B'_{n}(\sigma) := B_{n}(\sigma) - (-1)^{n/2}b_{n/2,n}$.

Similar to the argument above, one has that the rim hooks of leg length 0 on the top row together contribute $B'_{n-l}(\sigma')$. But since $l$ is odd, the rim hooks of leg length 0 on the bottom row now collectively contribute 
$$-\sum_{k=0}^{n/2 - l} (-1)^k b_{k,n-l}(\sigma').$$ 
Similarly, the remaining rim hooks of leg length 1 now contribute 
$$-\sum_{k=n/2 - l + 1}^{(n-l)/2 - 1} (-1)^{k} b_{k,n-l}(\sigma').$$ 
Altogether, this gives $B'_{n-l}(\sigma') - B'_{n-l}(\sigma') = 0$, as desired.
\end{proof}

\noindent Note that Equation~(\ref{eq:odd}) and Equation~(\ref{eq:even}) are recovered by the orthogonality relations, i.e., $\langle 1_{S_n}, H_n \rangle_{S_n} = n! = \langle 1_{S_n}, B_n \rangle_{S_n}$ where $1_{S_n}$ denotes the trivial representation of $S_n$.
In Section~\ref{sec:proofs} we use these characters to prove Theorem~\ref{thm:main} along with another new albeit less surprising combinatorial result concerning odd-cycle-intersecting permutations.

\section{Proof of Theorem~\ref{thm:main}}\label{sec:proofs}

We say that a family of permutations of $S_n$ is \emph{odd-cycle-intersecting} if the difference of any two elements in the family has an odd cycle. If $n$ is odd, then $S_n$ itself is odd-cycle-intersecting, so we assume $n$ is even. 
\begin{theorem}
Let $n$ be even. If $\mathcal{F} \subseteq S_n$ is odd-cycle-intersecting, then $|\mathcal{F}| \leq (n-1)!$. Equality holds if and only if $\mathcal{F}$ is a double translate of the Young subgroup $S_{n-1} \times S_1$.
\end{theorem}
\begin{proof}
 Let $H_\lambda$ be the hook product of $\lambda \vdash n$, and let $E_\lambda := (H_\lambda^{-1} \chi^\lambda(\sigma \pi^{-1})_{\pi,\sigma})$ denote the orthogonal projection onto the $\lambda$-isotypic component of the space of real-valued functions on $S_n$. Let $\mathcal{E}_n \subseteq S_n$ be the set of permutations that have no odd cycle. Any odd-cycle-intersecting family is an independent set of $\mathrm{Cay}(S_n, \mathcal{E}_n)$.
By Theorem~\ref{thm:2row}, the $n! \times n!$ matrix $\mathrm{EC}_n := \sum_{k=0}^{n/2} (-1)^k H_{b_{k,n}}E_{b_{k,n}}$ is a weighted adjacency matrix for $\mathrm{Cay}(S_n, \mathcal{E}_n)$ with constant row sum. It is easily verified that $H_{(n)}$ and $H_{(n-1,1)}$ are the largest and second largest hook products arising in the summation, thus the greatest eigenvalue is $n!$ and the least eigenvalue is $-n!/(n-1)$. Delsarte--Hoffman gives 
\[
	|S| \leq n! \frac{n!/(n-1)}{n! + n!/(n-1)} = (n-1)!.
\]
The Young subgroups $S_{n-1} \times S_1$ all have an odd (singleton) cycle, so the bound is met with equality. The later implies that the characteristic vector of any maximum odd-cycle-intersecting family lies in the direct sum of the $(n)$ and $(n-1,1)$-isotypic components. Ellis, Friedgut, and Pilpel~\cite{EllisFP11} showed that the only binary vectors with $(n-1)!$-many 1's that lie in this subspace are double translates of $S_{n-1} \times S_1$, completing the proof.
\end{proof}
\noindent Here, the combinatorial result is a bit unsurprising, as $\mathrm{EC}_n$ is also a weighted subgraph of \emph{the derangement graph}, i.e., the Cayley graph on $S_n$ generated by its derangements. It is well known that its maximum independent sets are the double translates of $S_{n-1} \times S_1$ (see~\cite{GodsilMeagher}, for example), and it has already been observed that its independence number does not rise when passing to many natural conjugacy-closed subgraphs (see~\cite{KuLW16,FilmusL24}). Despite appearances, the foregoing is not an analogue of the \emph{oddtown problem} for sets~\cite{BF92}.

\medskip

Before we begin the proof of Theorem~\ref{thm:main}, we recall some basic facts about Sylow $2$-subgroups of the symmetric group. We refer the reader to~\cite{Renteln25,WildonBlog2} for more details.
By definition, the order of each element in a Sylow $p$-subgroup $P_n \leq S_n$ is a power of $p$. For $p=2$, it is well known that $P_n$ can be identified with the automorphism group of a forest $F$ of $m$ complete binary trees on $n$ leaves, i.e., 
$$P_n \cong \mathrm{Aut}(F) \cong \mathrm{Aut}(T_1) \times \cdots \times \mathrm{Aut}(T_m) \subseteq S_n.$$ 
In particular, write $n = 2^{h_1} + \cdots + 2^{h_m}$ where $h_1 > \cdots  > h_m \geq 0$. Each $h_i$ denotes the height of $T_i$, and $\mathrm{Aut}(T_i)$ is the $h_i$-fold iterated wreath product of the form $C_2 \wr \cdots \wr C_2$ where $C_2$ is the cyclic group of order 2. Legendre's formula gives the exponent of the largest power of a prime $p$ that divides $n!$, which for $p=2$ gives us 
$$|P_n| = 2^{\sum_{i \ge 1} \lfloor n/2^i \rfloor}.$$ 
The exponent is maximized when $n$ is a power of 2, say $n=2^\ell$, in which case we have 
$$|P_n| = 2^{1 + 2 + \cdots + 2^{\ell-1}} =  2^{2^\ell - 1} =  2^{n - 1}.$$
We collect the foregoing observations into the following proposition.
\begin{proposition}\label{prop:sylow}
Any Sylow 2-subgroup $P_n \leq S_n$ is a 2-singular family of permutations such that $|P_n| \leq 2^{n-1}$, and equality holds if and only if $n$ is a power of 2. 
\end{proposition}
\noindent We are now in a position to give a proof of Theorem~\ref{thm:main}.
\begin{proof}[Proof of Theorem~\ref{thm:main}]
\noindent As before, let $H_\lambda$ be the hook product of $\lambda \vdash n$, and let $E_\lambda = (H_\lambda^{-1} \chi^\lambda(\sigma \pi^{-1}))_{\pi,\sigma}$ be the orthogonal projection onto the $\lambda$-isotypic component of the space of real-valued functions on $S_n$. Define the $n! \times n!$ matrix 
\[
	\mathrm{OC}_n := \sum_{k=0}^{n-1} H_{h_{k,n}} E_{h_{k,n}} = (\chi^\lambda(\sigma \pi^{-1})_{\pi,\sigma}).
\]
Theorem~\ref{thm:main} implies that $\mathrm{OC}_n' := \mathrm{OC}_n - 2^{n-1}I$ is a weighted adjacency matrix for the 2-regular graph of $S_n$ with constant row and column sum. In particular, we have
\[
	\mathrm{OC}_n' = \sum_{k=0}^{n-1} (n \cdot (n-k-1)! \cdot k! - 2^{n-1})~E_{h_{k,n}} + \sum_{\substack{\lambda \vdash n \\ \lambda \text{ not a hook}}} -2^{n-1}E_\lambda,
\]
thus $-2^{n-1}$ is the least eigenvalue and $n! - 2^{n-1}$ is the largest eigenvalue of $\mathrm{OC}_n'$. Delsarte--Hoffman gives
\[
	|S| \leq n! \cdot \frac{2^{n-1}}{n! - 2^{n-1} + 2^{n-1}} = 2^{n-1}.
\]
Now let $n = 2^\ell$ and let $S = P_n \leq S_n$ be a Sylow 2-subgroup. Proposition~\ref{prop:sylow} implies that the Sylow 2-subgroups and their double-translates, i.e., $\sigma P_n \pi$ for some $\sigma, \pi \in S_{n}$, are largest independent sets of the 2-regular graph of $S_{n}$. This completes the proof.
\end{proof}
We now collect a few corollaries of Theorem~\ref{thm:main}.
\noindent Letting $J$ be the $n! \times n!$ all-ones matrix, it is easy to see that $\mathrm{OC}_n - J$ is, remarkably,  an \emph{integral} feasible solution to the SDP~(\ref{eq:sdp}) that minimizes the objective function when $n$ is a power of 2. Since $\Gamma_{2'}(S_{n})$ is a normal Cayley graph, Theorem~\ref{thm:main} and a well-known result of Schrijver implies that $\vartheta(\overline{\Gamma_{2'}(S_{n})}) = n!/2^{n-1}$ when $n$ is a power of 2. Another corollary of Theorem~\ref{thm:main} a simple proof of a result of Giannelli on the induced representation of the trivial representation of a Sylow 2-subgroup (e.g., see \cite[Prop 3.5]{StaceyO23}).
\begin{corollary}[Giannelli~\cite{Giannelli17}]
Let $n$ be a power of 2 and let $P_n \leq S_n$ be a Sylow $2$-subgroup of $S_n$. Then the non-trivial hook-shaped irreducible representations of $S_n$ are not irreducible constituents of the permutation representation of $S_n$ acting on $S_n/P_n$.
\end{corollary}
\begin{proof}
In the proof of Theorem~\ref{thm:main}, equality is met in the Delsarte--Hoffman bound; therefore, the characteristic vector of any Sylow 2-subgroup $P_n$ is orthogonal to the direct sum of the irreducibles indexed by non-trivial hooks, i.e., no non-trivial hook appears in the induced representation $1 \uparrow_{P_n}^{S_n}$.
\end{proof}
\noindent

\section{Steinberg-like Characters}\label{sec:stlike}

The significance of the sum of hooks character has not gone unnoticed by group theorists, as it shares some remarkable similarities with a distinguished irreducible character of $\mathrm{GL}_{n,q} := \mathrm{GL}_n(\mathbb{F}_q)$ known as \emph{the Steinberg character}. For any finite field $\mathbb{F}_q$ of characteristic $p$, it is well-known that the Steinberg character vanishes on the $p$-singular elements of $\mathrm{GL}_{n,q}$ and its degree is the size of a largest Sylow $p$-subgroup of $\mathrm{GL}_{n,q}$ (see~\cite{Humphreys87}, for example). The sum of hooks character indeed mimics properties of the Steinberg character in characteristic 2, which raises the question of whether there are other characters that mirror the Steinberg character for various primes. In pursuit of this, Malle and Zalesski~\cite{MalleZ20} proposed the following.

\begin{definition}[$\mathrm{Syl}_p$-regular, $p$-vanishing, Steinberg-like~\cite{MalleZ20}]
Let $G$ be a finite group, let $p$ be a prime, and let $S \leq G$ be its Sylow $p$-subgroup. A character $\chi$ is \emph{$\mathrm{Syl}_p$-regular} if the restriction of $\chi$ to its Sylow $p$-subgroup $S$ is the character of the regular representation of $S$.  If, in addition, $\chi$ is \emph{p-vanishing}, i.e., $\chi(g) = 0$ for all $p$-singular elements $g \in G$, then we say that $\chi$ is \emph{Steinberg-like}.
\end{definition}
\noindent As noted in~\cite{MalleZ20}, the sum of hooks character $H_n$ is Steinberg-like with respect to $p=2$. 
\noindent It is not hard to see that Theorem~\ref{thm:main} generalizes to any Steinberg-like character of a group.
\begin{theorem}\label{thm:main2}
If $\chi$ is Steinberg-like character of a group $G$ with respect to some prime $p$ and $\langle \chi, 1 \rangle = 1$, then its Sylow $p$-subgroups are maximum $p$-singular families of $G$.
\end{theorem}
\begin{proof}
Let $P \leq G$ be a Sylow $p$-subgroup of $G$. Consider the graph $\Gamma_{p'}(G)$ and consider the matrix $A_\chi$ induced by $\chi$, i.e., $(A_\chi)_{\sigma \pi} := \chi(\sigma \pi^{-1})$. Since $\chi$ is Steinberg-like, it follows that $\tilde{A}_\chi := A_\chi - |P| I$ is a weighted adjacency matrix for $\Gamma$.  

Since $\langle \chi, 1 \rangle = 1$, the greatest eigenvalue of $\tilde{A}_\chi $ is $|G|-|P|$. Since $\chi(1) \geq \chi(g)$, the least eigenvalue of $\tilde{A}_\chi$ is $-|P|$. The Delsarte--Hoffman bound gives
\[
	|S| \leq |G| \frac{|P|}{|G|-|P| + |P| } = |P|,
\]
hence the size of a $p$-singular family $S \subseteq G$ is no greater than $|P|$. Sylow $p$-subgroups are $p$-singular, which completes the proof.
\end{proof}
\noindent This gives a combinatorial way of ruling out Steinberg-like characters with respect to $p$ with the additional condition that $\langle \chi, 1 \rangle = 1$. In particular, if there exists a $p$-singular family that is larger than a Sylow $p$-subgroup, then such characters cannot exist. 

Building off a series of papers, Malle and Zalesski~\cite{MalleZ20} essentially classified the Steinberg-like characters of finite simple groups. They do not exist for many natural classes of groups for primes $p \neq 2$, and in fact, the existence of the sum of hooks character for $n = 2^\ell$ (restricted to $A_n$) is what prevented a full classification~\cite{MalleZ20}. Malle and Zalesski leave it as an open question whether Steinberg-like characters of $S_n$ exist for $p=2$ and $n \neq 2^\ell$. We note that the SDP~(\ref{eq:sdp}) becomes a \emph{linear program} if further one constrains the matrices to belong to the Bose--Mesner algebra of the group conjugacy class association scheme of any finite group $G$, so in principle, determining if there exists a Steinberg-like character of $G=S_n$ for $p=2$ such that $\langle \chi, 1\rangle = 1$ could be resolved for modest $n$ by solving a relatively small integer linear program. For our purposes, we note that the integrality constraint is not necessary. Indeed, the existence of a conic combination $f$ of irreducible characters vanishing on $p$-singular elements and satisfying both $\langle f , 1 \rangle = 1$ and $f(1) = |P|$ implies that the Sylow $p$-subgroups are maximum $p$-singular families of $G$. This raises the question of whether such relaxations of Steinberg-like characters exist for simple groups, which we leave as an open question.


When a finite group $G$ admits an (irreducible) Steinberg character $\mathrm{St}$, it is natural to ask what this implies about the size of a maximum $p$-singular family. It is not difficult to show that $\mathrm{St}$ gives a non-trivial bound on the maximum size of a $p$-singular family of $\mathrm{GL}_{n,q}$ when $q$ is a power of $p$. For this, we require another bound, incomparable to Lov\'asz-$\vartheta$ and Delsarte--Hoffman, based on matrix rank.
\begin{theorem}[Haemers~{\cite{Haemers79}}]
For any graph $X = (V,E)$ and a field $K$, a $|V| \times |V|$ matrix $M$ with entries in $K$ is said to \emph{fit} $X$ if $m_{i,i} = 1$ for all $i \in V$ and $m_{i,j} = m_{j,i} = 0$ whenever $ij \not \in E$. Let $M_K(X)$ to be the set of all matrices over $K$ that fit $X$. Then $\alpha(X) \leq \mathcal{H}_K(X) := \min\{\mathrm{rank}_K(M) : M \in M_K(X)\}$. 
\end{theorem}

\begin{theorem}
Let $q=p^a$ for some prime $p$ and $a \in \mathbb{N}$. Then for all $n \geq 2$, we have 
$$\alpha (\Gamma_{p'}(\mathrm{GL}_{n,q})) \leq q^{n^2-n} = o(|\mathrm{GL}_{n,q}|).$$
\end{theorem}
\begin{proof}
 It is well-known that the absolute value of the character value of $\mathrm{St}$ on an element $g$ equals the order of a Sylow subgroup of the centralizer of $g$ if $g$ has order prime to $p$, and is zero if the order of g is divisible by $p$~(see \cite{Humphreys87}, for example); therefore, the orthogonal projection matrix $E_{\mathrm{St}}$ onto the $\mathrm{St}$-isotypic component is a $|\mathrm{GL}_n(q))| \times |\mathrm{GL}_n(q))|$ matrix that fits the graph $\Gamma_{p'}(\mathrm{GL}_n(q))$ after scaling so that the diagonal entries are 1.  Then the degree of the Steinberg character is the order of a Sylow $p$-subgroup, which is isomorphic to its subgroup of unitriangular matrices. Thus we have $\mathrm{St}(1) = q^{\binom{n}{2}}$ which shows $\mathrm{rank}(E_{\mathrm{St}}) = \chi(1)^2 = q^{n^2-n}$. Haemers' bound gives 
$$\alpha(\Gamma_{p'}(\mathrm{GL}_n(q))) \leq q^{n^2-n} = o(|\mathrm{GL}_n(q)|)$$
for fixed $q$ and $n \rightarrow \infty$. Note $|\mathrm{GL}_n(q)| = q^{\binom{n}{2}}\prod_{i=1}^n (q^i-1)$ and $\prod_{i=1}^n (q^i-1) \approx q^{\binom{n+1}{2}}$.
\end{proof}

\noindent Similar bounds also hold for other finite classical groups with a Steinberg representation. 

\section{Conjectures and Future Work}
Our main result showed that K\"orner's conjecture is true for $C = 2$, but we believe that that the constant $C$ can be improved. 
We also believe the Sylow 2-subgroups of $S_n$ are precisely the extremal families for all $n \geq 2$. 
\begin{conjecture}\label{con:main}
The Sylow 2-subgroups and their double-translates are the extremal even-cycle-intersecting families of $S_n$ for all $n \geq 2$. 
\end{conjecture}
\noindent For odd primes, it is unclear what the bound should be for the size of a maximum $p$-singular family of $S_n$, but for finite general linear groups, we conjecture the following.
\begin{conjecture}
Let $p \in \mathbb{N}$ be a prime and $q$ be a power of $p$. Then the Sylow $p$-subgroups and their double-translates are the extremal $p$-singular families of $\mathrm{GL}_{n,q}$ for all $n \geq 2$. 
\end{conjecture}

It may be interesting to generalize Theorem~\ref{thm:2row} to other families of representations of $S_n$ along the lines of Regev's generalization of Theorem~\ref{thm:hooks}~\cite{Regev13}. Using connections between the hook and two-row Specht modules in characteristic 2 (see~\cite{Kapadia25,Murphy82,Peel71}, for example), it may be possible to give a proof of Theorem~\ref{thm:2row}  akin to Wildon's~\cite{WildonBlog}. 

Finally, is there a compelling ``oddtown" analog for permutations? For example, what is the largest size of a family $\mathcal{F} \subseteq S_n$ such that $|\pi | \equiv 1 \mod 2$ for all $\pi  \in \mathcal{F}$ and $|\sigma\pi^{-1}| \equiv 0 \mod 2$ for all $\sigma, \pi \in \mathcal{F}$?

\subsection*{Acknowledgements}

We thank J\'anos K\"orner, William Linz, and Yuval Wigderson for helpful discussions that improved the exposition of the paper, and John Byrne for bringing reversing families of permutations to our attention.

\bibliographystyle{plain}
\bibliography{sample.bib}

\end{document}